\documentclass[11pt]{amsart}

\usepackage[a4paper, total={6.5in, 9.2in}]{geometry}

\usepackage[english]{babel}

\usepackage{amsmath,amsthm,amsfonts,amssymb}

\usepackage[all]{xy}

\usepackage[vcentermath]{youngtab}

\usepackage{color}

\usepackage{xparse}

\newcommand{\Dotfill}{\leaders\hbox to 8pt{\hss.\hss}\hfill}

\usepackage{tikz}
\tikzset{
v/.style={draw, fill, circle, minimum size=1.5mm, inner sep=0},
b/.style={draw , circle, minimum size=2.5mm, inner sep=.5mm},
e/.style={very thick},
vs/.style={draw, fill, circle, minimum size=1mm, inner sep=0},
bs/.style={draw, fill, circle, minimum size=1.5mm, inner sep=0mm},
es/.style={thick}
}
\newlength{\nodeheight}
\newlength{\nodewidth}
\usetikzlibrary{arrows,matrix,positioning}

\usepackage[colorlinks, linkcolor=blue,citecolor=blue,urlcolor=blue]{hyperref}
\usepackage{aliascnt}

\newaliascnt{thmauto}{thmcounter}

\newaliascnt{Defauto}{thmcounter}

\newaliascnt{exauto}{thmcounter}

\newaliascnt{lemauto}{thmcounter}

\newaliascnt{propauto}{thmcounter}

\newaliascnt{corauto}{thmcounter}

\newaliascnt{remauto}{thmcounter}

\newaliascnt{convauto}{thmcounter}

\renewcommand{\sectionautorefname}{Section}
\renewcommand{\subsectionautorefname}{Section}
\renewcommand{\subsubsectionautorefname}{Section}

\newtheorem{thmA}{Theorem}

\newtheorem{thm}[thmauto]{Theorem}
\newtheorem{prop}[propauto]{Proposition}
\newtheorem{lem}[lemauto]{Lemma}
\newtheorem{cor}[corauto]{Corollary}

\theoremstyle{definition}
\newtheorem{Def}[Defauto]{Definition}
\newtheorem{rem}[remauto]{Remark}

\DeclareMathOperator{\coker}{coker}

\DeclareMathOperator{\Ind}{Ind}
\DeclareMathOperator{\Res}{Res}

\DeclareMathOperator{\Pa}{P}

\providecommand{\CP}{\ensuremath\mathcal C_{\Pa}}

\providecommand{\N}{\ensuremath\mathbb N}

\providecommand{\C}{\ensuremath\mathbb C}

\title{Corrigendum to ``Representation stability for diagram algebras'' [J. Algebra 638 (2024), 625--669]}
\date{}
\author{Peter Patzt}
\address{Department of Mathematics, University of Oklahoma}
\email{ppatzt@ou.edu}

\begin{document}

\renewcommand{\sectionautorefname}{Section}
\renewcommand{\subsectionautorefname}{Section}
\renewcommand{\subsubsectionautorefname}{Section}

\maketitle

After the publication of \cite{diagalgrepstab}, we noticed a serious mistake in the treatment of the partition algebra. In particular, we assumed in the proof of Lemma 4.9 that the trivial $\Pa_n$-module is $\Pa_n(\emptyset)$ but it is actually indexed by the partition $(n)$. In this erratum, we fix this mistake and give a complete proof of Theorem C of \cite{diagalgrepstab}:

\setcounter{thmA}{2}
\begin{thmA}\label{thmA:P} Under the assumption that $R = \mathbb C$ and $\delta\in \mathbb C\setminus\N$,
a finitely presented $\CP$-module gives rise to a representation stable sequence of representations of the partition algebras.
\end{thmA}

In order to this, we have to change the definition of representation stability to the following:

\begin{Def} A $\CP$-module $V$ is \emph{representation stable} if it satisfies the following three conditions.
\begin{description}
\item[Injectivity] The map $\phi_n\colon V_n \to V_{n+1}$ is injective for all large enough $n\in\mathbb N$.
\item[Surjectivity] The induced map $\Ind_{\Pa_n}^{\Pa_{n+1}} \phi_n\colon \Ind_{\Pa_n}^{\Pa_{n+1}}V_n \to V_{n+1}$ is surjective for all large enough $n\in \mathbb N$.
\item[Multiplicity stability] If we write
\[ V_n \cong \bigoplus_{0\le k \le n}\bigoplus_{\lambda\in \overline \Pi_{n-k}} \mathfrak \Pa_n(\lambda[n-k])^{\oplus \alpha_{\lambda,k,n}},\]
then for fixed $\lambda$ and $k$, $m_{\lambda,k,n}$ is independent of $n$ for all large enough $n\in \mathbb N$.\footnote{Recall that $\overline \Pi_{n-k}$ is the set of partition $\lambda = ( \lambda_1 \ge \dots \ge \lambda_\ell)$ such that $n-k \ge |\lambda| + \lambda_1$, i.e.\ such that $\lambda[n-k] = (n-k - |\lambda|\ge \lambda_1\ge \dots \ge \lambda_\ell)$ is a partition.}
\end{description}
We say that $V$ is \emph{representation stable from $N$ on} if the conditions of injectivity, surjectivity, and multiplicity stability are satisfied for all $n\ge N$. If such an $N$ exists, $V$ is \emph{uniformly representation stable}.
\end{Def}

\begin{rem}
Despite not mentioned in Theorem A, B, and C, we actually prove uniform representation stability in Theorem 4.13, Theorem 4.14, and \autoref{partitionthm}. In particular, there are only finitely many partitions $\lambda$ and $k \in \mathbb N$ such that $\alpha_{\lambda,k,n}$ are nonzero for some $n\in \N$.
\end{rem}

\subsection*{Background on Kronecker Coefficients}

For $\lambda,\mu,\nu \in \Pi_n$, the \emph{Kronecker coefficient} $g_{\lambda,\mu,\nu}$ is defined as
\[ g_{\lambda,\mu,\nu} := [\mathfrak S_n(\lambda) \otimes \mathfrak S_n(\mu), \mathfrak S_n(\nu)].\]
Murnaghan's Theorem says that for any partitions $\lambda,\mu,\nu$ (not necessarily of the same size), the sequence
\[ g_{\lambda[n],\mu[n],\nu[n]},\]
where $n$ necessarily has to be at least $\max(|\lambda|+\lambda_1, |\mu|+\mu_1, |\nu|+\nu_1)$, stabilizes for $n$. The \emph{reduced Kronecker coeffcients} are defined as
\[ \bar g_{\lambda,\mu,\nu} :=  \lim_{n\to \infty}  g_{\lambda[n],\mu[n],\nu[n]}.\]

Recall that the Littlewood--Richardson coefficients are defined as
\[ c^\nu_{\lambda,\mu} := [\Ind^{\mathfrak S_n}_{\mathfrak S_l \times \mathfrak S_m} \mathfrak S_l(\lambda) \otimes \mathfrak S_m(\mu), \mathfrak S_n(\nu)]\]
for $\lambda \in \Pi_l, \mu\in \Pi_m, \nu\in \Pi_n$ and $l+m=n$. We will also use the notation
\[ c^\nu_{\alpha,\beta,\gamma} = [\Ind^{\mathfrak S_n}_{\mathfrak S_a \times \mathfrak S_b \times \mathfrak S_c} \mathfrak S_a(\alpha) \otimes \mathfrak S_b(\beta) \otimes \mathfrak S_c(\gamma), \mathfrak S_n(\nu)] = \sum_\xi c^\nu_{\alpha, \xi}c^\xi_{\beta ,\gamma}\]
for $\alpha\in\Pi_a, \beta\in \Pi_b, \gamma\in \Pi_c, \nu \in \Pi_n$ and $a+b+c=n$.

Bowman--De~Visscher--Orellana give the following formula for reduced Kronecker coefficients in terms of Littlewood--Richardson coefficients and Kronecker coefficients.

\begin{thm}[{\cite[Corollary 4.5]{BDVO}}]
\[ \bar g_{\lambda,\mu,\nu} = \sum_{\substack{\alpha,\beta,\gamma,\\ \pi, \rho,\sigma}} c^\lambda_{\alpha,\beta,\pi}\,\,c^\mu_{\alpha,\gamma,\rho}\,c^\nu_{\beta,\gamma,\sigma}\,g_{\pi,\rho,\sigma}\]
\end{thm}

This formula has three consequences that we are going to use. We also give earlier references.

\begin{cor}[{\cite[Theorem 2.9.22]{jameskerber}}]\label{triangle}
\[ \bar g_{\lambda,\mu,\nu} = 0\quad \text{ if $|\lambda|+|\mu| < |\nu|$.}\]
\end{cor}

\begin{cor}[\cite{Lit}]\label{K=LR}
\[ \bar g_{\lambda,\mu,\nu} = c_{\lambda,\mu}^{\nu}\quad\text{if $|\lambda|+|\mu| = |\nu|$.}\]
\end{cor}

\begin{cor}[{\cite[Corollary 5.1]{BDVO}}]\label{formula}
\[ \bar g_{\lambda,\mu,(k)} = \sum_{\substack{\alpha, \pi, \\ k_1+k_2+|\pi| = k}} c^\lambda_{\pi,\alpha, (k_1)} \,c^\mu_{\pi,\alpha ,(k_2)} \]
\end{cor}

\subsection*{Stability of reduced Kronecker coefficients} We prove the stability of a certain sequence of reduced Kronecker coefficients here.

\begin{lem}\label{LRflip}\begin{enumerate}
\item $c^{\lambda[n]}_{\xi, (n-|\xi|)} =1$ always implies $c^\xi_{\lambda,(|\xi|-|\lambda|)} =1$. 
\item $ c^\xi_{\lambda,(|\xi|-|\lambda|)} =1$ implies $c^{\lambda[n]}_{\xi, (n-|\xi|)}=1$ if $n\ge |\lambda| + \xi_1$. 
\item In particular, $c^\xi_{\lambda,(|\xi|-|\lambda|)} =c^{\lambda[n]}_{\xi, (n-|\xi|)}$ if $n\ge |\lambda| + \xi_1$.
\end{enumerate}
\end{lem}

\begin{proof}
Pieri's rule tells us that $c^{\lambda[n]}_{\xi,(n-|\xi|)} = 1$ if and only if we can obtain $\xi$ from $\lambda[n]$ by removing at most one box per column. If we remove from $\xi$ one box from from every column that was not used to obtain $\xi$ from $\lambda[n]$, we obtain $\lambda$. This implies that $c^\xi_{\lambda,(|\xi|-|\lambda|)}=1$. This proves (a).

Vice versa, if $c^\xi_{\lambda,(|\xi|-|\lambda|)}=1$, $\xi$ can be obtained from $\lambda$ by adding at most one box per column. Then we can add to $\xi$ a box in all the columns of $\xi$ that were not used to obtain $\xi$ from $\lambda$ to obtain $\lambda[|\lambda|+\xi_1]$. We could have of course added more boxes in the first row and so get that $c^{\lambda[n]}_{\xi,(n-|\xi|)} = 1$ for all $n\ge |\lambda|+\xi_1$. This proves (b).

(a) and (b) combined shows that $c^{\lambda[n]}_{\xi,(n-|\xi|)} = 1$ if and only if $c^\xi_{\lambda,(|\xi|-|\lambda|)}=1$ or all $n\ge |\lambda|+\xi_1$. Pieri's rule implies that both Littlewood--Richardson coefficients are $0$ if they are not $1$. This implies (c).
\end{proof}

\begin{prop}\label{Kroneckerstab}
Fix partitions $\lambda, \mu$ and a number $k \in \N$. The sequence $(\bar g_{\lambda[n], \mu,(n-k)})_n$ is independent of $n\ge |\lambda| + |\mu|$.
\end{prop}

\begin{proof}
From \autoref{formula}, we know that
\[ \bar g_{ \lambda[n],\mu,(n-k)} = \sum_{\substack{\alpha, \pi, \\ k_1+k_2+|\pi| = n-k}} c^{\lambda[n]}_{\pi,\alpha, (k_1)} \,c^\mu_{\pi,\alpha, (k_2)}. \]
Because $\mu$ is fixed there is a fixed finite set $S$ of tuples $(\pi,\alpha,k_2,\xi)$ such that $c^\mu_{\pi,\alpha, (k_2)} \neq 0$ and $c^\xi_{\pi,\alpha}\neq 0$. We observe that for every $\xi$ appearing in $S$, we know $|\xi| = |\pi|+|\alpha| \le |\mu|$. This implies that if $n\ge |\lambda| + |\mu|$ that $c^{\lambda[n]}_{\xi, (n-|\xi|)} = c^\xi_{\lambda,(|\xi|-|\lambda|)}$ by \autoref{LRflip}(c). Thus we get
\[ \bar g_{ \lambda[n],\mu,(n-k)} = \sum_{(\pi,\alpha,k_2,\xi)\in S} c^{\lambda[n]}_{\xi, (n-|\xi|)} c^\xi_{\pi,\alpha} c^\mu_{\pi,\alpha,(k_2)} =  \sum_{(\pi,\alpha,k_2,\xi)\in S} c^{\xi}_{\lambda, (|\xi|-|\lambda|)} c^\xi_{\pi,\alpha} c^\mu_{\pi,\alpha,(k_2)}\]
for $n\ge |\lambda|+|\mu|$, which is clearly independent of $n$.
\end{proof}

\begin{cor}\label{size}
$\bar g_{\lambda[n], \mu,(n-k)} =0$ if $|\lambda|>|\mu|$.
\end{cor}

\begin{proof}
As in the proof of \autoref{Kroneckerstab},
\[ \bar g_{\lambda[n], \mu,(n-k)}  = \sum_{(\pi,\alpha,k_2,\xi)\in S} c^{\lambda[n]}_{\xi, (n-|\xi|)} c^\xi_{\pi,\alpha} c^\mu_{\pi,\alpha,(k_2)} .\]
This can only be nonzero if there is a tuple $(\pi,\alpha,k_2,\xi)\in S$ such that $c^{\lambda[n]}_{\xi, (n-|\xi|)}=1$. By \autoref{LRflip}(a), this implies that $c^\xi_{\lambda,(|\xi|-|\lambda|)} =1$. But this is in contradiction to $|\xi| \le |\mu| < |\lambda|$.
\end{proof}

\subsection*{Corrected proofs} We now give corrections to the incorrect Proposition 4.8, Proposition 4.12, and Theorem 4.15 of \cite{diagalgrepstab}.

The following proposition replaces \cite[Proposition 4.8]{diagalgrepstab}.

\begin{prop}\label{4.8}
Let $\mu\in \Pi_{\Pa_m}$ and $n-i \ge|\mu|+ \mu_1$, then the following statements hold.
\begin{enumerate}
\item $[  \tau_{n,m}\Pa_{n}(\mu[n-i]), \Pa_m(\lambda)]= 0$ if $|\lambda| < m-i$.
\item $[  \tau_{n,m}\Pa_{n}(\mu[n-i]), \Pa_m(\lambda)]= 0$ if $|\lambda| = m-i < |\mu|$.
\item $[ \tau_{n,m}\Pa_{n}(\mu[n-i]), \Pa_m(\lambda)]=  \delta_{\lambda \mu}$ if $|\lambda| = m-i = |\mu|$.
\item $[ \tau_{n,m}\Pa_{n}(\mu[n-i]), \Pa_m(\lambda)]$ is independent of $n\ge |\lambda| + |\mu| + i$.
\end{enumerate}
\end{prop}

\begin{proof}
Recall that the trivial $\Pa_{n-m}$-module $\C$ is indexed by $(n-m)$. The branching rule  \cite[Corollary 3.11]{diagalgrepstab} implies that
\[ [\Res_{\Pa_m\otimes \Pa_{n-m}}^{\Pa_{n}} \Pa_{n}(\mu[{n}-i]), \Pa_m(\lambda) \otimes \C] = \bar g_{\lambda, (n-m),\mu[{n}-i]}.\]
 \cite[Lemma 4.6]{diagalgrepstab}  gives
\[ [ \tau_{{n},m}\Pa_{n}(\mu[{n}-i]), \Pa_m(\lambda)] = \bar g_{\lambda,(n-m),\mu[{n}-i]}.\]

If $|\lambda|<m-i$, then
\[ |\lambda| + n-m < n-i = |\mu[n-i]|\]
and thus  $\bar g_{\lambda,(n-m),\mu[{n}-i]}=0$ by \autoref{triangle}. This implies (a).

If $|\lambda|=m-i$, then
\[ |\lambda| + n-m = n-i = |\mu[n-i]|\]
and thus  $\bar g_{\lambda,(n-m),\mu[{n}-i]}=c_{\lambda,(n-m)}^{\mu[n-i]}$ by \autoref{K=LR}. By Pieri's rule, this is nonzero if and only if $\lambda$ can be obtained from $\mu[n-i]$ by removing $n-m$ boxes, none of which are in the same column. In particular, the first row of $\mu[n-i]$ which has $n-i-|\mu|$ boxes, has to have at least $n-m$ boxes. This implies that $|\mu| \le m-i = |\lambda|$ and so (b).

If $|\lambda|=m-i =|\mu|$, Pieri's rule tells us that $\lambda$ can be obtained from $\mu[n-i]$ by removing $n-m$ boxes, none in the same column. Because the first row of $\mu[n-i]$ has exactly $n-i-|\mu| = n-m$ boxes, $\mu = \lambda$. This proves (c).

In general,  $\bar g_{\lambda, (n-m), \mu[n-i]}$ is independent of $n \ge |\lambda|+|\mu| +i$ by \autoref{Kroneckerstab}. This proves (d).
\end{proof}

The following proposition replaces  \cite[Proposition 4.12]{diagalgrepstab}

\begin{prop}\label{4.12}
Let $V$ be a $\CP$-module finitely generated in degrees $\le g$ and $\Pa_n(\lambda[n-i])$ is a constituent of $V_n$, then $ i \le 2g$ and $|\lambda| \le g$.
\end{prop}

\begin{proof}
It suffices to prove that if $\Pa_n(\lambda[n-i])$ is a constituent of $M(m)_n$, then $ i \le 2m$ and $|\lambda| \le m$.

Recall that 
\[M(m)_n = \Pa_n \otimes_{\Pa_{n-m}} \mathbb C = \Ind_{\Pa_m\otimes \Pa_{n-m}}^{\Pa_n} \Pa_m\otimes \mathbb C.\]
Thus we need to consider the constituents $\Pa_n(\lambda[n-i])$ of
\[ \Ind_{\Pa_m\otimes \Pa_{n-m}}^{\Pa_n} \Pa_m(\mu)\otimes \Pa_{n-m}(n-m)\]
for all partitions $\mu\in \Pi_{\Pa_m}$. From Frobenius reciprocity and the branching rule we see that
\begin{align*}
 &[ \Ind_{\Pa_m\otimes \Pa_{n-m}}^{\Pa_n} \Pa_m(\mu)\otimes \Pa_{n-m}(n-m),\Pa_n(\lambda[n-i])] \\
=& [ \Res_{\Pa_m\otimes \Pa_{n-m}}^{\Pa_n}\Pa_n(\lambda[n-i]),  \Pa_m(\mu)\otimes \Pa_{n-m}(n-m)] \\
 = &\bar g_{\mu,(n-m),\lambda[n-i]}.
\end{align*}
So it suffices to show that $\bar g_{\mu,(n-m),\lambda[n-i]}=0$ if $i>m+|\mu|$ or $|\lambda|> |\mu|$ because $|\mu|\le m$. 

The inequality $i>m+|\mu|$ is equivalent to $|\mu| +(n-i)< (n-m)$ which by \autoref{triangle} implies that $\bar g_{\mu,(n-m),\lambda[n-i]}=0$.

The inequality $|\lambda|> |\mu|$ implies that $\bar g_{\mu,(n-m),\lambda[n-i]}=0$ by \autoref{size}.
\end{proof}

We now give a correction of \cite[Theorem 4.15]{diagalgrepstab}. 

\begin{thm}\label{partitionthm}
Let $V$ be a $\CP$-module finitely generated in degrees $\le g$ with relations in degrees $\le r$, then $V$ is representation stable from $\le 3g + \max(2g,2r)$ on.
\end{thm}

\begin{proof}
We first prove injectivity and surjectivity. We need to prove that
\[ K_n = \ker( V_n \longrightarrow V_{n+1}) \quad \text{and}\quad C_n = \coker( \Ind_{\Pa_n}^{\Pa_{n+1}} V_n \longrightarrow V_{n+1})\]
vanish for $n\ge 3g+\max(2g,2r)$. Because $\tau_{n,a}$ is exact, we see that 
\[ \tau_{n,a} K_n = \ker( \tau_{n,a} V_n \longrightarrow \tau_{n,a} V_{n+1}) \quad \text{and}\quad \tau_{n+1,a}C_n = \coker( \tau_{n+1,a} \Ind_{\Pa_n}^{\Pa_{n+1}} V_n \longrightarrow \tau_{n+1,a} V_{n+1}).\]
By  \cite[Proposition 4.4]{diagalgrepstab}, $V$ has stability degree $\le \max(2g,2r)$ and so the composition
\[ \tau_{n,a} V_n \longrightarrow \tau_{n,a} V_{n+1} \longrightarrow \tau_{n+1,a} V_{n+1}\]
from \cite[Equation (4.1)]{diagalgrepstab} is an isomorphism for all $a$ and all $n\ge \max(2g,2r)+a$, so in particular, $\tau_{n,a} K_n = 0$ for all $a$ and all $n\ge \max(2g,2r)+a$. We can write $\bar\phi_n$ also as the composition
\[   \bar\phi_n\colon \tau_{{n},a} V_{n}  \longrightarrow \tau_{{n}+1,a} \Ind_{\Pa_{n}}^{\Pa_{{n}+1}} V_{n}  \longrightarrow \tau_{{n}+1,a} V_{{n}+1}\]
as in \cite[Equation (4.2)]{diagalgrepstab}. Thus $\tau_{{n}+1,a} C_{n} = 0$ for all $a$ and all $n\ge \max(2g,2r)+a$. 

Assume to the contrary of $K_n =0$ that $\Pa_n(\lambda[n-i])$ is a constituent of $K_n$. It then must also be a constituent of $V_n$ as $K_n \subset V_n$. By  \autoref{4.12}, this means that $i\le 2g$ and $|\lambda| \le g$ as $V$ is finitely generated in degrees $\le g$. Set $a=|\lambda|+i \le 3g$ to get that
\[ [\tau_{n,a} \Pa_n(\lambda[n-i]), \Pa_a(\lambda)]=1\]
by \hyperref[prop:partition tau]{\autoref{4.8}(c)}. As $\tau_{n,a}$ is exact, this means that $\Pa_a(\lambda)$ is a constituent of $\tau_{n,a} K_{n}$. But we have seen above, the latter is zero as
\[ n \ge 3g+\max(2g,2r) \ge \max(2g,2r)+a.\]

Likewise, let us assume to the contrary of $C_n = 0$ that $\Pa_{n+1}(\lambda[n+1-i])$ is a constituent of $C_n$. Because $C_n$ is a quotient of $V_{n+1}$, we get $i\le 2g$ and $|\lambda|\le g$ by \autoref{4.12}. Again, set  $a=|\lambda|+i \le 3g$ to get 
\[ [\tau_{n+1,a} \Pa_{n+1}(\lambda[n+1-i]) , \Pa_a(\lambda)] = 1\]
by \hyperref[prop:partition tau]{\autoref{4.8}(c)}. Then, $\Pa_a(\lambda)$ is a constituent of $\tau_{n+1,a} C_{n}$ which again is zero as
\[ n \ge 3g+\max(2g,2r) \ge\max(2g,2r)+ a.\]

These contradictions prove that $K_n$ and $C_n$ are in fact zero for $n\ge 3g+ \max(2g,2r)$.

We now prove multiplicity stability. Let us decompose
\[ V_n = \bigoplus_{0\le i \le n}\bigoplus_{\mu \in \overline\Pi_{n-i}} \Pa_n(\mu[n-i])^{\oplus \alpha_{\mu,i,n}}\]
into its simple constituents, where $\alpha_{\mu,i,n}$ is the multiplicity of $\Pa_n(\mu[n-i])$ in $V_n$.
By  \autoref{4.12}, we can assume $ i \le 2g$ and $|\mu| \le g$. We will prove that $\alpha_{\mu,i,n}$ is independent of $n\ge 3g + \max(2g,2r)$ by induction over $i$ going down and then $|\mu|$ going up. For $i>2g$ or $|\mu|<0$, there is nothing to prove as $\alpha_{\mu,i,n}=0$ by \autoref{4.12}. This establishes the base of the induction. For the induction step, fix $0\le j\le 2g$ and $\lambda\in \overline\Pi_{n-j}$ with $|\lambda|\le g$ and assume the assertion has been proved for all $(\mu,i)$ with $i>j$ or $i=j$ and $|\mu|<|\lambda|$.

Set $m := |\lambda| + j \le 3g$. We will count the multiplicity of $\Pa_m(\lambda)$ in $\tau_{n,m} V_n$. The following equation uses that $\tau_{n,m}$ is additive.
\begin{align}
&\big[ \tau_{n,m} V_n\,,\,\Pa_m(\lambda)\big]  \label{eqPatauV1}\\
= &\phantom{+}\hspace{.6em} \sum_{0\le i <j}\sum_{\mu\in \overline\Pi_{n-i}} \big[ \tau_{n,m} \Pa_n(\mu[n-i])\,,\,\Pa_m(\lambda) \big]\cdot  \alpha_{\mu,i,n}\label{eqPatauV2}\\
&+ \phantom{\sum_{j<i \le 2g}}\sum_{\substack{\mu\in \overline\Pi_{n-j}\\|\mu|>|\lambda|}} \big[ \tau_{n,m} \Pa_n(\mu[n-j]) \,,\,\Pa_m(\lambda)\big]\cdot \alpha_{\mu,j,n}\label{eqPatauV3a}\\
&+ \phantom{\sum_{j<i \le 2g}}\sum_{\substack{\mu\in \overline\Pi_{n-j}\\|\mu|=|\lambda|}} \big[ \tau_{n,m} \Pa_n(\mu[n-j]) \,,\,\Pa_m(\lambda)\big]\cdot \alpha_{\mu,j,n}\label{eqPatauV3}\\
&+ \phantom{\sum_{j<i \le 2g}}\sum_{\substack{\mu\in \overline\Pi_{n-j}\\|\mu|<|\lambda|}} \big[ \tau_{n,m} \Pa_n(\mu[n-j]) \,,\,\Pa_m(\lambda)\big]\cdot \alpha_{\mu,j,n}\label{eqPatauV3b}\\
&+ \sum_{j<i \le 2g}\sum_{\mu\in \overline\Pi_{n-i}} \big[ \tau_{n,m} \Pa_n(\mu[n-i])\,,\,\Pa_m(\lambda) \big] \cdot  \alpha_{\mu,i,n}\label{eqPatauV4}
\end{align}
From   \cite[Proposition 4.4]{diagalgrepstab} we know that $V$ has stability degree $\le \max(2g,2r)$ and so \eqref{eqPatauV1} is independent of 
\[n\ge\max(2g,2r)+3g \ge \max(2g,2r)+m.\] 

\autoref{4.8}(a) says that \eqref{eqPatauV2} is zero because $i<j$ implies
\[  |\lambda| =m -j < m-i.\]

\autoref{4.8}(b) says that \eqref{eqPatauV3a} is zero because $|\lambda| =m-j<|\mu|$.

 \autoref{4.8}(c) says that \eqref{eqPatauV3} is $\alpha_{\lambda,j,n}$.

\autoref{4.8}(d) and the induction hypothesis provide that \eqref{eqPatauV3b} and \eqref{eqPatauV4} are independent of 
\[ n \ge  \max( |\lambda|+|\mu|+i, 3g + \max(2g,2r)). \]

Because all terms in this equation are finite and $ |\lambda|+|\mu|+i \le 4g$, we get that $\alpha_{\lambda,j,n}$ is independent of $n \ge   3g + \max(2g,2r)$.
\end{proof}

\bibliographystyle{halpha}
\bibliography{repstab}

\end{document}